\documentclass[a4paper,12pt]{amsart}
\usepackage{amsmath,amsthm,amssymb}
\usepackage{hyperref}

\allowdisplaybreaks[1]

\newcommand{\ti}{\widetilde}
\newcommand{\al}{\alpha}
\newcommand{\be}{\beta}
\newcommand{\la}{\lambda}
\newcommand{\ity}{\infty}
\newcommand{\C}{\mathbb{C}}

\newcommand{\N}{\mathbb{N}}

\newcommand{\B}{\Big}

\numberwithin{equation}{section}
\newtheorem{theorem}{Theorem}[section]
\newtheorem{lemma}[theorem]{Lemma}
\newtheorem{corollary}[theorem]{Corollary}
\theoremstyle{remark}
\newtheorem{remark}[theorem]{Remark}

\newtheorem{definition}[theorem]{Definition}

\begin{document}
\title[semigroups and their dynamics]{escaping set and julia set of transcendental semigroups}
\author[D. Kumar]{Dinesh Kumar}
\address{Department of Mathematics, University of Delhi,
Delhi--110 007, India}

\email{dinukumar680@gmail.com }
\author[S. Kumar]{Sanjay Kumar}

\address{Department of Mathematics, Deen Dayal Upadhyaya College, University of Delhi,
New Delhi--110 015, India }

\author[K. K. Poon]{Kin Keung Poon}
\address{Department of Mathematics and Information Technology, Hong Kong Institute of Education,
Hong Kong}

\email{kkpoon@ied.edu.hk}

\begin{abstract}
We discuss the dynamics of  semigroups of transcendental entire functions using Fatou-Julia theory and provide a condition for the complete invariance of  escaping set and Julia set of  transcendental semigroups. Results regarding  limit functions and postsingular set are also derived. In addition, classes of hyperbolic and postsingularly bounded   transcendental semigroups are given. We also provide certain criterion for non existence of   wandering domains of  transcendental semigroups. 
\end{abstract}

\keywords{Semigroup, normal family, completely invariant, hyperbolic}

\subjclass[2010]{37F10, 30D05}

\maketitle

\section{introduction}\label{sec1}

Let $f$ be a transcendental entire function and for $n\in\N,$ let $f^n$ denote the $n$-th iterate of $f.$ The set $F(f)=\{z\in\C : \{f^n\}_{n\in\N}\,\text{ is normal in some neighborhood of}\, z\}$ is called the Fatou set of $f$ or the set of normality of $f$ and its complement $J(f)$ is called the Julia set of $f$. An introduction to the basic properties of these sets can be found in \cite{bergweiler}. The escaping set of $f$ denoted by $I(f)$ is the set of points in the complex plane that tend to infinity under iteration of $f$. 
The set $I(f)$ was introduced for the first time by Eremenko \cite {e1} who established that $I(f)$ is non empty, and each component of $\overline{I(f)}$ is unbounded.

A complex number $\xi\in\C$ is called a \emph{critical value} of a transcendental entire function $f$ if there exist some $w\in\C$ such that $f(w)=\xi$ and $f'(w)=0.$ Here $w$ is called a \emph{critical point} of $f$ and its image under $f$ is a critical value of $f.$ A complex number  $\zeta\in\C$ is an \emph{asymptotic value} of a transcendental entire function $f$ if there exist a curve $\Gamma$ tending to infinity such that $f(z)\to \zeta$ as $z\to\ity$ along $\Gamma.$  Recall the
 Eremenko-Lyubich class
 \[\mathcal{B}=\{f:\C\to\C\,\,\text{transcendental entire}: \text{Sing}{(f^{-1})}\,\text{is bounded}\},\]
where Sing($f^{-1}$) is the set of critical values and asymptotic values of $f$ and their finite limit points. Each $f\in\mathcal B$ is said to be of \emph{bounded type}. Moreover,  if $f$ and $g$ are of bounded type, then so is $f\circ g$ \cite{berg2}. 
The definitions for  critical point, critical value and asymptotic value of a transcendental semigroup $G$ were provided in \cite{dinesh1}. 

Two functions $f$ and $g$ are called permutable if $f\circ g=g\circ f.$
Fatou   proved that if $f$ and $g$ are two rational functions which are permutable, then $F(f)=F(g),$ see \cite{beardon}. This was an important result  that motivated the dynamics of composition of complex functions. Analogous results for transcendental entire functions are still not known, though it holds in some very special cases   \cite[Lemma 4.5]{baker2}.
 The authors  in \cite{dinesh3} have constructed several  examples where the dynamical behavior of $f$ and $g$ differ to a large extent, from the dynamical behavior of their compositions. Using approximation theory of entire functions, they have shown the existence of entire functions $f$ and $g$ having infinite number of domains satisfying various properties and relating it to their compositions. They explored and enlarged all the maximum possible ways of the solution in comparison to the past result worked out. It would be interesting to explore such relations in the context of transcendental semigroups and its constituent elements.
The authors considered the relationship between Fatou sets and singular values of transcendental entire functions $f, g$ and their compositions \cite{dinesh2}. They provided several conditions under which Fatou sets of $f$ and $f\circ g$ coincide and also considered the relation between the singular values of $f, g$ and their compositions. 

Recently, the dynamics of composite of two or more  complex functions have been studied by many authors.
 The seminal work in this direction was done by Hinkkanen and Martin \cite{martin} related to semigroups of rational functions. In their  papers, they extended the classical theory of the dynamics  associated to the iteration of a rational function of one complex variable to a more general setting of an arbitrary semigroup of rational functions. Several results were extended to semigroups of transcendental entire functions in \cite{{cheng}, {dinesh1}, {dinesh5}, {dinesh6}, {poon1}, {zhigang}}. 
It should be noted that Sumi  has done an extensive work in the semigroup theory of rational functions and holomorphic maps. He has written a series of papers, for instance, \cite{sumi1, sumi2}. 

A transcendental semigroup $G$ is a semigroup generated by a family of transcendental entire functions $\{f_1,f_2,\ldots\}$ with the semigroup operation being functional composition. Denote  the semigroup by $G=[f_1,f_2,\ldots].$ Thus  each $g\in G$ is a transcendental entire function and $G$ is closed under composition. The semigroup $G$ is called \emph{abelian} if $f_i\circ f_j=f_j\circ f_i,$ for all $i, j\in\N.$
The  Fatou set $F(G)$ of a transcendental semigroup $G$, is the largest open subset of $\C$ on which the family of functions in $G$ is normal and the Julia set $J(G)$ of $G$ is the complement of $F(G),$ that is, $J(G)=\ti\C\setminus F(G).$ The semigroup generated by a single function $g$ is denoted by $[g].$ 
The following definitions are well known in  transcendental semigroup theory.
\begin{definition}\label{sec1,defn1}
Let $G$ be a transcendental semigroup. A set $W$ is said to be \emph{forward invariant} under $G$ if $g(W)\subset W$ for all $g\in G$ and $W$ is said to be  \emph{backward invariant} under $G$ if $g^{-1}(W)=\{w\in\C:g(w)\in W\}\subset W$ for all $g\in G.$ Furthermore, $W$ is called \emph{completely invariant} under $G$ if it is both forward and backward invariant under $G.$ 
\end{definition}
For a transcendental semigroup $G,$ $F(G)$ is forward invariant and $J(G)$ is backward invariant, \cite[Theorem 2.1]{poon1}.

The contrast between the dynamics of a semigroup and those of a single function is that the semigroup dynamics is more complicated.  For instance, $F(G)$ and $J(G)$ need not be completely invariant and $J(G)$ may have interior points without being entire complex plane $\C,$ \cite{martin}. In \cite{dinesh5}, the authors  extended  the dynamics of a transcendental entire function on its escaping set to the dynamics of semigroups of transcendental entire functions on their escaping sets and initiated the study of escaping sets of semigroups of transcendental entire functions. 

In this paper, we have considered the dynamics of semigroups of transcendental entire functions using Fatou-Julia theory. We have provided some condition for the complete invariance of escaping set and Julia set of a class of transcendental semigroups. Some results on limit functions and postsingular set have been discussed. A class of hyperbolic and postsingularly bounded  transcendental semigroups  domains have been provided. We also provide some criterion for non existence of   wandering domains of  transcendental semigroups. 

\section{theorems and their proofs}\label{sec2}

Recall \cite[Definition 2.1]{dinesh5}, for a transcendental semigroup $G$, the escaping set of $G,$ denoted by $I(G)$ is defined as
\[I(G)=\{z\in\C\, |\,\text{every sequence in}\,G\,\text{has a subsequence which diverges to infinity at}\, z\}.\]
The following result provides backward invariance of $I(G)$

\begin{theorem}\label{sec2,thm1}
Let $G=[g_1, g_2\ldots]$ be an abelian transcendental semigroup. Then  $I(G)$ is backward invariant under $G.$
\end{theorem}

\begin{proof}
We need to show that $g^{-1}(I(G))\subset I(G)$ for all $g\in G.$ Suppose $w\notin I(G).$ Then there exist a  sequence $\{f_n\}$ in $G$ whose no subsequence  diverges to $\ity$ at $w.$ Consequently, all subsequences of $\{f_n\}$ are eventually bounded at $w.$  Let $g\in G$ and consider the sequence $\{f_n\circ g\}.$  Since all subsequences of $\{f_n\}$ are eventually bounded at $w$ and $g$ is continuous, therefore, all subsequences of $\{g\circ f_n\}$ are bounded at $w.$ As $G$ is abelian, $f_n\circ g=g\circ f_n$ for all $n\in\N.$ Thus all subsequences of the sequence $\{f_n\circ g\}$  are eventually bounded at $w.$ In other words, all subsequences of the sequence $\{f_n\}$ are eventually bounded at $g(w)$ and so $g(w)\notin I(G)$ for all $g\in G.$ Hence, $g^{-1}(I(G))\subset I(G)$ for all $g\in G$ and this completes the proof of the theorem.
\end{proof}

\begin{remark}\label{sec2,rem3}
In \cite[Theorem 4.1]{dinesh5}, it was shown that for a transcendental semigroup $G, I(G)$ is  forward invariant under $G.$ Theorem \ref{sec2,thm1}, in particular, establishes that for an abelian transcendental semigroup $G, I(G)$ is completely invariant.
\end{remark}
Even if a transcendental semigroup $G$ is non abelian, still $I(G)$ (and hence $\overline{I(G)}$) can be completely invariant. We first prove an elementary lemma.

\begin{lemma}\label{sec2,lem1}
Let $f$ be a transcendental entire function.
Then the closure of any forward invariant subset of $\C$ is forward invariant.
\end{lemma}

\begin{proof}
Suppose $A\subset\C$ is forward invariant and let $z\in\overline A.$ Then there exist a sequence $\{z_n\}$ in $A$ such that $z_n\to z.$ By the continuity of $f,$ $f(z_n)\to f(z).$ As $A$ is forward invariant, so $f(z)\in \overline A$ and hence $\overline A$ is forward invariant. 
\end{proof}

\begin{theorem}\label{sec2,thm2}
Let $G=[g_1,\ldots, g_n]$ be a finitely generated transcendental semigroup in which each $g_i,\;1\leq i\leq n$ is of bounded type. Then $\overline {I(G)}$ is completely invariant under $G.$
\end{theorem}

\begin{proof}
As $I(G)$ is forward invariant under $G,$ its closure is forward invariant under $G$ by Lemma \ref{sec2,lem1}. We now show that $\overline{I(G)}$ is backward invariant under $G,$ that is, $g^{-1}(\overline{I(G)})\subset \overline{I(G)}$ for all $g\in G.$ Let $w\notin\overline{I(G)}.$ Then there exist a neighborhood $U$ of $w$ such that $U\cap\overline {I(G)}=\emptyset.$ Since $G$ is of bounded type, $I(G)\subset J(G)$ and this  implies that $U\subset F(G)\subset F(g)$ for all $g\in G.$ As each $g\in G$ is of bounded type,  $I(g)\subset J(g)$ \cite{EL}, and hence  $U\cap I(g)=\emptyset$ for all $g\in G$ and so $w\notin \overline{I(g)}$ for all $g\in G.$ As a result,  $g(w)\notin \overline{I(g)}$ for all $g\in G.$ Therefore, $g(w)\notin\overline{I(G)}$ for all $g\in G$ and hence  $w\notin g^{-1}( \overline{I(G)})$ which proves the backward invariance of $\overline{I(G)}$ and hence the result.
\end{proof}

As a consequence, one obtains
\begin{corollary}\label{sec2,cor1}
Let $G=[g_1,\ldots, g_n]$ be a finitely generated transcendental semigroup in which each $g_i,\;1\leq i\leq n$ is of bounded type. Then $J(G)$ is completely invariant under $G.$
\end{corollary}
\begin{proof}
By \cite[Theorem 4.5(ii)]{dinesh5}, $J(G)=\overline{I(G)}$ and so Theorem \ref{sec2,thm2} gives the desired result.
\end{proof}

The next result gives a relation between the escaping set of a semigroup and its generators.

\begin{theorem}\label{sec2,thm3}
Let $G=[g_1,\ldots, g_k]$ be a finitely generated abelian transcendental semigroup. Then $I(G)=\cap_{i=1}^{k}g_i^{-1}(I(G)).$
\end{theorem}

\begin{proof}
As $I(G)$ is forward invariant under $G,$  $g_i(I(G))\subset I(G),\;1\leq i\leq k$ which implies that $I(G)\subset \cap_{i=1}^{k}g_i^{-1}(I(G)).$
The argument for proving backward implication is similar to the one used in the proof of Theorem \ref{sec2,thm1}. We give it for the sake of  completeness.
Suppose $w\notin I(G).$ Then there is a  sequence $\{f_n\}$ in $G$
such that all its subsequences  are eventually bounded at $w.$  For $g_i\in G,\;1\leq i\leq k$  consider the sequence $\{f_n\circ g_i\}.$  Since  all subsequences of $\{f_n\}$ are eventually bounded at $w,$ and $g_i$ is continuous therefore, all subsequences of $\{g_i\circ f_n\}$ are bounded at $w.$ As $G$ is abelian, $f_n\circ g_i=g_i\circ f_n$ for all $n\in\N$ and so all subsequences of the sequence $\{f_n\circ g_i\}$  are eventually bounded at $w.$ Thus all subsequences of the sequence $\{f_n\}$ are eventually bounded at $g_i(w)$ and hence  $g_i(w)\notin I(G)$ for all $1\leq i\leq k.$ This shows that $w\notin g_i^{-1}(I(G))$ for all $1\leq i\leq k$ and  $\cap_{i=1}^{k}g_i^{-1}(I(G))\subset I(G)$. This completes the proof of the theorem.
\end{proof}

Recall, \cite[p.\ 61]{Hua} for a transcendental meromorphic function $g,$ a function $\psi(z)$ is a limit function of $(g^n)$ on a  component $V\subset F(g)$ if there is some subsequence of $(g^n)$ which converges locally uniformly on $V$ to $\psi.$ Denote by $\mathfrak{L}(U)$ all such limit functions. The following result gives a criterion for connectivity of $J(G):$

\begin{theorem}\label{sec2,thm4}
Let $G=[g_1,\ldots, g_n]$ be a finitely generated abelian transcendental semigroup. If for each $i,\;1\leq i\leq n,$ every Fatou component of  $g_i$ is bounded and $\ity$ is not a limit function of any sequence in $G$ in a component of $F(G),$ then $J(G)\subset\C$ is connected. 
\end{theorem}

\begin{proof}
Observe that for all $g\in G,$ every component of $F(g)$ is bounded. As $\ity$ is not a limit function of any sequence in $G$ in a component of $F(G),$ from \cite[Theorem 4.1]{dinesh1}, every component of $F(G)$ is simply connected and therefore, $J(G)\subset\C$ is connected using \cite[Theorem 2.11]{dinesh6}.
\end{proof}

\begin{remark}\label{sec2,rem1}
For all $g\in G,$ every Fatou component of $g$ is simply connected, and therefore, $J(g)\subset\C$ is connected.
\end{remark}

\begin{theorem}\label{sec2,thm5}
Let $f$ be a transcendental entire function of period of $c$ and let $g=f^m+c$ for some $m\in\N.$ Let $U\subset F(f)$ be an invariant component and $\phi$ be a limit function of $(f^n)$ on $U.$  Then $\phi$ is a limit function of $((f\circ g)^n)$ on $U$ and $\phi+c$ is a limit function of $((g\circ f)^n)$ on $U.$
\end{theorem}

\begin{proof}
Observe that for all $n\in\N, (f\circ g)^n(z)=f^{n(m+1)}(z)$ and $(g\circ f)^n(z)=f^{n(m+1)}(z)+c.$ As $\phi\in\mathfrak{L}(U)$ for $(f^n),$ there exist a subsequence $(f^{n_i})$ of    $(f^n)$ with $\lim f^{n_i}(z)=\phi(z)$ on $U.$ For $z\in U,\, \lim(f\circ g)^{n_i}(z)=\lim f^{n_i(m+1)}(z)=\phi(z)$ which implies that  $\phi$ is a limit function of $((f\circ g)^n)$ on $U.$ 
Similarly, it can be seen that $\lim(g\circ f)^{n_i}(z)=\phi(z)+c$\, on $U$ and hence the result.
\end{proof}

\section{postsingular set}\label{sec3}

For an entire function $f,$ let $E(f)=\cup_{n\geq 0}f^{n}(\text{Sing}(f^{-1}))$ and $E^\prime(f)$ be the derived set of $E(f),$ that is, the set of finite limit points of $E(f).$ Then the union of $E(f)$ and $E'(f)$ is the postsingular set  $\mathcal P(f),$ that is, 
\[\mathcal P(f)=\overline{\B(\bigcup_{n\geq 0}f^n(\text{Sing}(f^{-1}))\B)}.\]

 For a transcendental semigroup $G$, let 
\[
\mathcal P(G)=\overline{\B(\bigcup_{f\in G}{\text{Sing}}(f^{-1})\B)}.
\]

The first theorem of this section gives  relation between postsingular set of composite of two entire functions with those of its factors.
\begin{lemma}\label{sec2,lem2}
Let $f$ and $g$ be permutable transcendental entire functions of bounded type. Suppose $f(\text{Sing}(g^{-1}))\subset\text{Sing}(g^{-1})$ and $g(\text{Sing}(f^{-1}))\subset\text{Sing}(f^{-1}).$ Then $\mathcal P(f\circ g)\subset \mathcal P(f)\cup\mathcal P(g).$ 
\end{lemma}

\begin{proof}
Since $f(\text{Sing}(g^{-1}))\subset\text{Sing}(g^{-1}),$ it is easy to see that $f^n(\text{Sing}(g^{-1}))\subset\text{Sing}(g^{-1})$ for all $n\in\N.$ Similarly, $g^n(\text{Sing}(f^{-1}))\subset\text{Sing}(f^{-1})$ for all $n\in\N.$ Using permutability of $f$ and $g,$ we have 
\begin{equation}\label{sec2,eq1}
\begin{split}\notag
\mathcal P(f\circ g)
&=\overline{\B(\bigcup_{n\geq 0}(f\circ g)^n(\text{Sing}({f\circ g})^{-1})\B)}\\
&\subset\overline{\B(\bigcup_{n\geq 0}(f^n(g^n(\text{Sing}(f^{-1})\cup f(\text{Sing}(g^{-1})))))\B)}\\
&=\overline{\B(\bigcup_{n\geq 0}(f^n(g^n(\text{Sing}(f^{-1}))))\cup (g^n(f^{n+1}(\text{Sing}(g^{-1}))))\B)}\\
&\subset\overline{\B(\bigcup_{n\geq 0}(f^n(\text{Sing}(f^{-1})))\B)}\cup\overline{\B(\bigcup_{n\geq 0}(g^n(\text{Sing}(g^{-1})))\B)}\\
&=\mathcal P(f)\cup\mathcal P(g).\qedhere
\end{split}
\end{equation}
\end{proof}

\begin{remark}\label{sec2,rem2}
It can be seen by an induction argument that if $g_1,\ldots, g_n$ are permutable transcendental entire functions of bounded type satisfying $g_i(\text{Sing}(g_j)^{-1})\subset\text{Sing}(g_j)^{-1},$ for all $1\leq i, j\leq n,\, i\neq j,$ then $\mathcal P(g_1\circ\cdots\circ g_n)\subset\bigcup_{i=1}^{n}\mathcal P(g_i).$
\end{remark}


Recall that an entire function $f$ is called \emph{hyperbolic} if the postsingular set $\mathcal P(f)$ is a compact subset of $F(f).$ For instance, $e^{\la z},\,0<\la<\frac{1}{e}$ are examples of hyperbolic entire functions. A transcendental semigroup $G$ is called hyperbolic if each $g\in G$ is hyperbolic \cite{dinesh5}. Also  an entire function $f$ is called \emph{postsingularly bounded} if the postsingular set $\mathcal P(f)$ is bounded. A transcendental semigroup $G$ is called postsingularly bounded if each $g\in G$ is postsingularly bounded \cite{dinesh5}.
The following result provides a class of hyperbolic transcendental semigroups:

\begin{theorem}\label{sec2,thm6}
Let $G=[g_1,\ldots, g_n]$ be a finitely generated abelian transcendental semigroup in which each generator is of bounded type. Suppose that for each $i, j,\;1\leq i, j\leq n,\, i\neq j$ $g_i(\text{Sing}(g_j)^{-1})\subset\text{Sing}(g_j)^{-1}.$ If the generators of $G$ are hyperbolic, then so is $G.$
\end{theorem}

\begin{proof}
 $\mathcal P(g_1\circ\cdots\circ g_n)\subset\bigcup_{i=1}^{n}\mathcal P(g_i)$ by Remark \ref{sec2,rem2}. Using the permutability of each $g_i,\,1\leq i\leq n,$ any $g\in G$ can be represented as $g=g_1^{l_1}\circ\cdots\circ g_n^{l_n}$ which implies that $\mathcal P(g)\subset\bigcup_{i=1}^{n}\mathcal P(g_i^{l_i}).$ Also for any entire function $f,$  $\mathcal P(f^k)=\mathcal P(f)$ \cite{baker1}, therefore, $\mathcal P(g)\subset\bigcup_{i=1}^{n}\mathcal P(g_i).$  Using the hyperbolicity of the generators, $\bigcup_{i=1}^{n}\mathcal P(g_i)$ is a compact subset of $\C$ which further implies that $\mathcal P(g)\subset\C$ is compact. As $F(g_i)=F(g)$\, $1\leq i\leq n,$\, \cite[Theorem 5.9]{dinesh1}, it is easily seen that $\mathcal P(g)\subset\bigcup_{i=1}^{n}\mathcal P(g_i)\subset\bigcup_{i=1}^{n}F(g_i)$ and hence $\mathcal P(g)\subset F(g).$ Thus $\mathcal P(g)$ is a compact subset of $F(g)$ and so $g$ is hyperbolic. This completes the proof of the theorem.
\end{proof}

\begin{remark}\label{sec2,rem4}
Under the hypothesis of Theorem \ref{sec2,thm6}, if the generators are postsingularly bounded, then so is $G.$
\end{remark}
 Given two permutable entire functions $f$ and $g$ of bounded type, we now establish a relation between $I(f), I(g)$ and $I(f\circ g).$

\begin{theorem}\label{sec2,thm6'}
 If $f$ and $g$ are two permutable entire functions of bounded type then  $I(f\circ g)\subset I(f)\cup I(g).$
\end{theorem}

\begin{proof}
Suppose $w\notin I(f)\cup I(g).$ Then $w\notin I(f)$ and $w\notin I(g).$ As a consequence of this we obtain $w\notin J(f)\cup J(g)$  (we may assume that  $w$ is neither a preperiodic point of $f$ nor of $g$).  As $f$ and $g$ are of bounded type, therefore $J(f)=J(g)=J(f\circ g)$ using \cite[Theorem 2.2(ii)]{dinesh2}. Moreover, as $f\circ g$ is of bounded type so $I(f\circ g)\subset J(f\circ g)$ and hence  $w\notin I(f\circ g).$ This completes the proof of the result.
\end{proof}


The following result shows that postsingularly bounded entire functions are closed under composition:
\begin{theorem}\label{sec2,thm8}
Let $f$ and $g$ be permutable transcendental entire functions of bounded type. If $f$ and $g$ are postsingularly bounded, then so is $f\circ g.$
\end{theorem}
\begin{proof}
Denote  $A=\text{Sing}(f^{-1}), B=\text{Sing}(g^{-1})$ and $C=\text{Sing}(f\circ g)^{-1}.$ It suffices to show that for all $w\in C, (f\circ g)^{n}(w)$ is bounded as $n\to\ity.$ Suppose that  for some $\al\in C, (f\circ g)^{n}(\al)\to\ity$ as $n\to\ity,$ that is, $\al\in I(f\circ g).$ As $C\subset A\cup f(B)$ \cite{berg2}, therefore, either $\al\in A$ or $\al\in f(B).$ Using Theorem \ref{sec2,thm6'} $I(f\circ g)\subset I(f)\cup I(g)$, accordingly we have the following four cases:
\begin{enumerate}
\item[(i)]   $\al\in A$ and $\al\in I(f).$ In this case, we arrive at   a contradiction to postsingular set of $f$  being bounded.\\
\item[(ii)] $\al\in A$ and $\al\in I(g).$ As $g$ is of bounded type, so $I(g)\subset J(g),$ and as $f\circ g=g\circ f,$ therefore $J(f)=J(g)$ \cite{dinesh1}. We obtain $\al\subset A\cap J(f).$ Since $\al\in\mathcal P(f)$ and $f$ being postsingularly bounded implies $\mathcal P(f)\subset F(f),$ that is, $\al\in F(f)$ a contradiction.\\
\item[(iii)] $\al\in f(B)$ and $\al\in I(f).$ As $f$ is of bounded type, $I(f)\subset J(f)$ and as argued above $J(f)=J(g).$ Also as $B\subset\mathcal P(g)\subset F(g),$ implies $f(B)\subset f(F(g)).$ This further implies $f(B)\cap J(g)\subset f(F(g))\cap J(g)\subset F(g)\cap J(g)=\emptyset,$ a contradiction.\\
\item[(iv)] $\al\in f(B)$ and $\al\in I(g).$ As $\al\in f(B),$ so $\al=f(\be)$ for some $\be\in B.$ As $f(\be)\in I(g),$  so $g^n(f(\be))\to\ity$ as $n\to\ity.$ Since $f$ and $g$ are permutable,  $f(g^n(\be))\to\ity$ as $n\to\ity.$ As $f$ is an entire function, $g^n(\be)$ must tend to $\ity$ as $n\to\ity,$ that is, $\be\in I(g),$ a contradiction to postsingular set of $g$ being bounded. 
\end{enumerate}
Thus in all the four cases we arrive at a contradiction. Hence $f\circ g$ is postsingularly bounded.
\end{proof}

\begin{remark}\label{sec2,rem5}
It can be seen by  induction  that if $g_1,\ldots, g_n$ are permutable entire functions of bounded type which are postsingularly bounded, then so is $g_1\circ\cdots\circ g_n.$
\end{remark}

The next result provides a class of postsingularly bounded transcendental semigroups:

\begin{theorem}\label{sec2,thm9}
Let $G=[g_1,\ldots, g_n]$ be a finitely generated abelian transcendental semigroup in which each generator is of bounded type. If the generators of $G$ are postsingularly bounded  then so is $G.$
\end{theorem}

\begin{proof}

$\mathcal P(g_1\circ\cdots\circ g_n)$ is bounded using Remark \ref{sec2,rem5}. From the permutability of each $g_i,\,1\leq i\leq n,$ any $g\in G$ can be represented as $g=g_1^{l_1}\circ\cdots\circ g_n^{l_n}.$ Denote  $A_1=\text{Sing}(g_1^{l_1})^{-1},\ldots, A_n=\text{Sing}(g_n^{l_n})^{-1}.$ It suffices to show that  $\mathcal P(g)$ is bounded. Suppose that for some $w\in \text{Sing}(g^{-1}), g^k(w)\to\ity$ as $k\to\ity,$ that is, $w\in I(g_1^{l_1}\circ\cdots\circ g_n^{l_n}).$ Then
\begin{equation}\label{sec2,eq3}
\begin{split}\notag
\text{Sing}(g)^{-1}
&=\text{Sing}(g_1^{l_1}\circ\cdots\circ g_n^{l_n})^{-1}\\
&\subset \text{Sing}(g_1^{l_1})^{-1}\cup g_1^{l_1}(\text{Sing}(g_2^{l_2})^{-1})\cup\cdots\cup g_1^{l_1}\circ\cdots\circ g_{n-1}^{l_{n-1}}(\text{Sing}(g_n^{l_n})^{-1})\\
&=A_1\cup g_1^{l_1}(A_2)\cup\cdots\cup g_1^{l_1}\circ\cdots\circ g_{n-1}^{l_{n-1}}(A_n).
\end{split}
\end{equation}
As $w\in \text{Sing}(g^{-1}),$ so let $w\in g_1^{l_1}\circ\cdots\circ g_{n-1}^{l_{n-1}}(A_n).$ This implies $w=g_1^{l_1}\circ\cdots\circ g_{n-1}^{l_{n-1}}(z)$ for some $z\in A_n.$ As $w\in I(g)=g_1^{l_1}\circ\cdots\circ g_n^{l_n}$ which further is contained in $\cup_{i=1}^{n}I(g_i^{l_i}).$ If $w\in I(g_n^{l_n})=I(g_n),$ then $g_n^k(w)\to\ity$ as $k\to\ity,$ that is $g_n^k(g_1^{l_1}\circ\cdots\circ g_{n-1}^{l_{n-1}}(z))\to\ity$ as $k\to\ity.$ Using the permutability of the generators we obtain, $g_1^{l_1}\circ\cdots\circ g_{n-1}^{l_{n-1}}(g_n^k(z))\to\ity$ as $k\to\ity.$ As $g_1^{l_1}\circ\cdots\circ g_{n-1}^{l_{n-1}}$ is an entire function, $g_n^k(z)$ must tend to $\ity$ as $k\to\ity,$ that is, $z\in I(g_n)\cap A_n$, which contradicts the hypothesis that $\mathcal P(g_n)$ is bounded. 
Now let $w\in I(g_i^{l_i})=I(g_i)$ for some $1\leq i<n.$ This implies $g_i^k(w)\to\ity$ as $k\to\ity,$ that is, $g_i^k(g_1^{l_1}\circ\cdots\circ g_{n-1}^{l_{n-1}}(z))\to\ity$ as $k\to\ity,$ which further implies that $g_1^{l_1}\circ\cdots\circ g_{n-1}^{l_{n-1}}(g_i^{k+l_i}(z))\to\ity$ as $k\to\ity.$ 
As $g_1^{l_1}\circ\cdots\circ g_{n-1}^{l_{-1}}$ is an entire function, $z\in I(g_i),$ that is, $z\in I(g_i)\cap A_n$ which further is contained in $J(g_i)\cap F(g_n).$ From \cite{dinesh1}[Theorem], $J(g_i)=J(g_j)$ for all $1\leq i, j\leq n,$ we arrive at a contradiction and hence $g$ is postsingularly bounded which in turn implies that $G$ is postsingularly bounded.
\end{proof}

%
%

Recall that a  component $U$ of $F(G)$ is called a \emph{wandering domain} of $G$ if the set $\{U_g:g\in G\}$ is infinite (where  $U_g$ is the component of $F(G)$ containing $g(U)$).
%
%
%
%
%
%

A criterion for non existence of wandering domains for an entire function of bounded type was given in \cite{berg3}.

\begin{lemma}\cite{berg3}\label{sec2,lem5'}
Let  $f$ be an entire function of bounded type for which $E^\prime(f)\cap J(f)$ is finite and consists only of rationally indifferent or repelling periodic points and preimages of such points. Then $f$ does not have wandering domains.
\end{lemma}
 We shall see that this result  gets generalized to semigroups. 
We first establish a lemma.
\begin{lemma}\label{sec2,lem5''}
Let $f$ and $g$ be permutable entire functions of bounded type. If both $\mathcal P(f)\cap J(f)$ and $\mathcal P(g)\cap J(g)$ are finite and consists only of rationally indifferent or repelling periodic points and preimages of such points (that is, if $f$ and $g$ have no wandering domains), then similar behavior also holds for $\mathcal P(f\circ g)\cap J(f\circ g).$
\end{lemma}

\begin{proof}
Using Theorem \ref{sec2,thm8}, $\mathcal P(f\circ g)\subset \mathcal P(f)\cup\mathcal P(g).$ Also from \cite[Theorem 2.2]{dinesh2}, $J(f)=J(g)=J(f\circ g).$ As a result, we obtain
\begin{equation}\label{sec2,eq2'}
 \begin{split}\notag
\mathcal P(f\circ g)\cap J(f\circ g)
&\subset(\mathcal P(f)\cup\mathcal P(g))\cap J(f)\\
&\subset(\mathcal P(f)\cap J(f))\cup(\mathcal P(g)\cap J(g)).\\ 
\end{split}
\end{equation}
It now follows that $\mathcal P(f\circ g)\cap J(f\circ g)$ is finite and consists only of rationally indifferent or repelling periodic points and preimages of such points, and hence $f\circ g$ has no wandering domains.
\end{proof}
 The next result rules out the existence of wandering domains for a certain class of transcendental semigroups. To prove the result, we require the following lemmas.

\begin{lemma}\cite[Lemma 4.6]{dinesh6}\label{sec2,lem3}
For a transcendental semigroup $G=[g_1,g_2,\ldots], \mathcal P(G)= \overline{\B(\bigcup_{g\in G}\mathcal P(g)\B)}.$
\end{lemma}

\begin{lemma}\cite[Theorem 4.2]{poon1}\label{sec2,lem4}
 For a transcendental semigroup $G=[g_1,g_2,\ldots],\,J(G)=\overline{\B(\bigcup_{g\in G}J(g)\B)}.$
\end{lemma}

\begin{theorem}\label{sec2,thm7}
Let $G=[g_1,\ldots, g_n]$ be a finitely generated abelian transcendental semigroup in which each generator is of bounded type. If for each $i,\, 1\leq i\leq n,$ $\mathcal P(g_i)\cap J(g_i)=\emptyset,$ then $G$ has no wandering domains.
\end{theorem}

\begin{proof}
Using Lemma \ref{sec2,lem5'}, for each $i,\, 1\leq i\leq n,$ $g_i$ has no wandering domains. As a consequence of Lemma \ref{sec2,lem5''}, each $g\in G$ does not contain wandering domains. In order to show that $G$ has no wandering domains, it suffices to show that $\mathcal P(G)\cap J(G)=\emptyset.$ Using Lemmas \ref{sec2,lem3} and \ref{sec2,lem4}, 
\begin{equation}\label{sec2,eq2}
 \begin{split}\notag
\mathcal P(G)\cap J(G)
&=\overline{\B(\bigcup_{g\in G}\mathcal P(g)\B)}\bigcap\overline{\B(\bigcup_{g\in G}J(g)\B)}\\
&=\overline{\B(\bigcup_{g\in G}(\mathcal P(g)\cap J(g))\B)}
\end{split}
\end{equation}
and hence we deduce that $\mathcal P(G)\cap J(G)=\emptyset.$ As a consequence of \cite[Theorem 5.23]{dinesh1}, $G$ has no wandering domains.
\end{proof}
The next result provides another criterion for the absence of wandering domains for a certain class of transcendental semigroups.

\begin{theorem}\label{sec2,thm7'}
Let $G=[g_1,\ldots, g_n]$ be a finitely generated abelian transcendental semigroup in which each generator is of bounded type. If $\mathcal P(G)\cap J(G)$ is finite and consists only of rationally indifferent or repelling periodic points and preimages of such points then $G$ has no wandering domains.
\end{theorem}

\begin{proof}
Using Lemmas \ref{sec2,lem3} and \ref{sec2,lem4}, we have
\begin{equation}\label{sec2,eq2''}
 \begin{split}\notag
\mathcal P(G)\cap J(G)
&=\overline{\B(\bigcup_{g\in G}\mathcal P(g)\B)}\bigcap\overline{\B(\bigcup_{g\in G}J(g)\B)}\\
&\supset\B(\bigcup_{g\in G}\mathcal P(g)\B)\bigcap\B(\bigcup_{g\in G}J(g)\B).\\
\end{split}
\end{equation}
As $\mathcal P(G)\cap J(G)$ is finite and consists only of rationally indifferent or repelling periodic points and preimages of such points, it follows that the same also holds for $\mathcal P(g)\cap J(h)$ for at most finitely many $g, h\in G,$ whereas, for infinitely many $g', h'\in G,$ $\mathcal P(g')\cap J(h')=\emptyset.$ Using \cite[Theorem 5.9]{dinesh1}, $J(f)=J(k)$ for all $f, k\in G.$ As a result, $\mathcal P(g)\cap J(g)$ is finite and consists only of rationally indifferent or repelling periodic points and preimages of such points for at most finitely many $g\in G,$ while for infinitely many  $h\in G,$ $\mathcal P(h)\cap J(h)=\emptyset.$  From Lemma \ref{sec2,lem5'}, for all $g\in G, g$ has no wandering domains and hence it follows that $G$ cannot have wandering domains. 
\end{proof}


\begin{thebibliography}{00}

\bibitem{baker1} I. N. Baker, Limit functions and sets of non-normality in iteration theory, Ann. Acad. Sci. Fenn. Ser. A. I. Math. \textbf{467} (1970), 1-11. 
\bibitem{baker2} I. N. Baker, Wandering domains in the iteration of entire functions, Proc. London Math. Soc. \textbf{49} (1984), 563-576.
\bibitem{beardon} A. F. Beardon, \emph{Iteration of rational functions}, Springer Verlag, (1991).
\bibitem{bergweiler} W. Bergweiler, Iteration of meromorphic functions, Bull. Amer. Math. Soc. \textbf{29} (1993), 151-188.
\bibitem{berg3} W. Bergweiler, M. Haruta, H. Kriete, H. G. Meier and N. Terglane, On the limit functions of iterates in wandering domain, Ann. Acad. Sci. Fenn. Ser. A. I. Math. \textbf{18} (1993), 369-375.
\bibitem{berg5} W. Bergweiler and A. Hinkkanen, On semiconjugation of entire functions, Math. Proc. Cambridge Philos. Soc. \textbf {126} (1999), 565-574. 

\bibitem{berg2} W. Bergweiler and Y. Wang, On the dynamics of composite entire functions. Ark. Math. \textbf{36} (1998), 31-39.

\bibitem{e1} A. E. Eremenko, On the iteration of entire functions, Ergodic Theory and Dynamical Systems, Banach Center Publications \textbf{23}, Polish Scientific Publishers, Warsaw, (1989), 339-345.
\bibitem{EL} A. E. Eremenko and M. Yu. Lyubich, Dynamical properties of some classes of entire functions, Ann. Inst. Fourier, Grenoble, \textbf{42} (1992), 989-1020.
\bibitem{F2}  P. Fatou, Sur l'it\'eration des fonctions transcendantes Enti\`eres, Acta Math. {\bf 47} (1926), no.~4, 337--370. 
\bibitem{martin} A. Hinkkanen and G. J. Martin, The dynamics of semigroups of rational functions I, Proc. London Math. Soc. (3) \textbf{73} (1996), 358-384.

\bibitem{Hua} X. H. Hua, C. C. Yang, \emph{Dynamics of transcendental functions}, Gordon and Breach Science Pub. (1998).

\bibitem{cheng} Z. G. Huang and T. Cheng, Singularities and strictly wandering domains of transcendental semigroups, Bull. Korean Math. Soc. (1) \textbf{50} (2013), 343-351.

\bibitem{dinesh1} D. Kumar and S. Kumar, The dynamics of semigroups of transcendental entire functions I, Indian J. Pure Appl. Math. \textbf{46} (2015), 11-24.
\bibitem{dinesh2} D. Kumar and S. Kumar, On dynamics of composite entire functions and singularities, Bull. Cal. Math. Soc. \textbf{106} (2014), 65-72.

\bibitem{dinesh3} D. Kumar, G. Datt and S. Kumar, Dynamics of composite entire functions, arXiv:math.DS/12075930, (2013) (accepted for publication in J. Ind. Math. Soc.)
\bibitem{dinesh5} D. Kumar and S. Kumar, The dynamics of semigroups of transcendental entire functions II, arXiv:math.DS/14010425, (2014), (accepted for publication in Indian J. Pure Appl. Math.)
\bibitem{dinesh6}  D. Kumar and S. Kumar, Semigroups of transcendental entire functions and their dynamics, arXiv:math.DS/14050224, (2014), (accepted for publication in Proc. Indian Acad. Sci. (Math. Sci.))

\bibitem{morosowa} S. Morosawa, Y. Nishimura, M. Taniguchi and T. Ueda, \emph{Holomorphic dynamics}, Cambridge Univ. Press, (2000).
\bibitem{poon1} K. K. Poon, Fatou-Julia theory on transcendental semigroups, Bull. Austral. Math. Soc. \textbf{58} (1998), 403-410.

\bibitem{poon2} K. K. Poon, Fatou-Julia theory on transcendental semigroups II, Bull. Austral. Math. Soc. \textbf{59} (1999), 257-262. 

\bibitem{R1} L. Rempe, On a question of Eremenko concerning escaping sets of entire functions, Bull. London Math. Soc. \textbf{39:4}, (2007), 661-666.
\bibitem{SZ} D. Schleicher and J. Zimmer, Escaping points of exponential maps, J. London Math. Soc. (2) \textbf{67} (2003), 380-400.

\bibitem{sumi1} H. Sumi, On dynamics of hyperbolic rational semigroups, J. Math. Kyoto Univ. \textbf{37} (1997), 717-733.
\bibitem{sumi2} H. Sumi, Random complex dynamics and semigroups of holomorphic maps, Proc. London Math. Soc. \textbf{102} (2011), 50-112.

\bibitem{zhigang} H. Zhigang, The dynamics of semigroups of transcendental meromorphic functions, Tsinghua Science and Technology, (4) \textbf{9} (2004), 472-474.





\end{thebibliography}
\end{document}